\documentclass[12pt]{amsart}
\title{Grzegorczyk logic unlocked}
\author{Wojciech Aleksander Wo\l oszyn}
\address[Wojciech Aleksander Wo\l oszyn]
{Mathematical Institute, University of Oxford, Andrew Wiles Building, Radcliffe Observatory Quarter, Woodstock Road, Oxford, OX2 6GG, United Kingdom \&\ St Hilda's College, Cowley Place, Oxford, OX4 1DY, United Kingdom}
\email{wojciech@wwresearch.org}
\urladdr{https://woloszyn.org}

\usepackage{latexsym}   
\usepackage{amsfonts}   
\usepackage{amsmath}    
\usepackage{amssymb}    
\usepackage{mathrsfs}   

\usepackage{ mathdots }

\usepackage[hidelinks]{hyperref} 
\usepackage[utf8]{inputenc} 
\usepackage{microtype} 
\usepackage{xurl} 

\usepackage[rgb,dvipsnames]{xcolor} 
\usepackage{tikz} 
\usetikzlibrary{backgrounds} 
\usetikzlibrary{arrows,arrows.meta,petri,topaths,positioning,shapes,shapes.misc,patterns,calc,decorations,decorations.pathreplacing,hobby} 
\usepackage{tikz-cd} 

\usepackage{adjustbox} 

\usepackage{caption} 

\usepackage{enumitem} 
\usepackage{booktabs} 

\usepackage[backend=biber,style=alphabetic,maxbibnames=15,maxcitenames=6,dateabbrev=false,autolang=hyphen]{biblatex}

\usepackage{mathtools} 

\usepackage{MathMacrosWAW} 

\usepackage{trees} 

\usepackage{multirow}

\renewcommand{\UrlFont}{\sffamily\smaller} 
\renewbibmacro{in:}{\ifentrytype{article}{}{\printtext{\bibstring{in}\intitlepunct}}} 
\DeclareFieldFormat{url}{{\UrlFont\url{#1}}} 
\DeclareFieldFormat{urldate}{
  (version \thefield{urlday}\addspace%
  \mkbibmonth{\thefield{urlmonth}}\addspace%
  \thefield{urlyear}\isdot)}
\DeclareFieldFormat{eprint:arxiv}{
\ifhyperref
    {\href{http://arxiv.org/abs/#1}{%
    arXiv\addcolon\nolinkurl{#1}}\iffieldundef{eprintclass}{}{\UrlFont{\mkbibbrackets{\thefield{eprintclass}}}}}
    {arXiv\addcolon\nolinkurl{#1}\iffieldundef{eprintclass}{}{\UrlFont{\mkbibbrackets{\thefield{eprintclass}}}}}}
\bibliography{HamkinsBiblio,MathBiblio,WoloszynBiblio}

\tikzset{>=Stealth,
  dot/.style={circle,draw,fill,inner sep=#1},
    dot/.default=.7pt
}

\begin{document}

\begin{abstract}
The article offers a fresh perspective on Grzegorczyk logic \theoryf{Grz}, introducing a simplified axiomatization and extending the analysis to its natural modal extensions, \theoryf{Grz.2} and \theoryf{Grz.3}. I develop a control statement theory for these logics, utilizing concepts such as buttons, switches, and ratchets. Through this framework, I establish that \theoryf{Grz.2} is characterized by finite Boolean algebras.
\end{abstract}

\maketitle
    \vspace*{.20cm}

\setcounter{tocdepth}{1}

\section{Introduction}

Despite its intimidating axiomatization, the Grzegorczyk logic has played an important role in modal logic and has been studied extensively by philosophers, mathematicians, and computer scientists alike~\cite{Maksimova2007:On-Modal-Grzegorczyk-Logic}. The axiom in question is stated as follows.
$$
\begin{array}{rl}
    \axiomf{Grz} & \necessary(\necessary( p \implies\necessary p )\implies p )\implies p.
\end{array}
$$

In this article, I bring Grzegorczyk logic closer to the reader by moving away from the complex axiom \axiomf{Grz} and developing tools and techniques that make this logic much easier to work with. To differentiate between the logic and the axiom, I will refer to Grzegorczyk logic as \(\theoryf{Grz}\). Similarly, I will use \(\theoryf{Grz.2}\) to refer to the modal logic generated by the axioms \(\axiomf{Grz}\) and \(\axiomf{.2}\), and \(\theoryf{Grz.3}\) to refer to the logic generated by \(\axiomf{Grz}\) and \(\axiomf{.3}\).

\medskip

I shift the focus to the new central notion of \emph{penultimacy}, a modal state where a current truth is just one step away from becoming irrevocably false. A light bulb still glowing might seem secure, but it could burn out any moment, and once it does, it will never shine again without replacement. Similarly, a business teetering on the edge of bankruptcy is operational today, but if it fails, the doors will close for good.

\newtheorem*{PenultimacyTheorem}{Penultimacy Theorem}%
\begin{PenultimacyTheorem}
\makeatletter\def\@currentlabelname{penultimacy theorem}\makeatother\label{PenultimacyTheorem}
    Grzegorczyk logic is the smallest reflexive and transitive normal modal logic in which every contingency possibly attains a penultimate truth-value. 
\end{PenultimacyTheorem}

The \nameref{PenultimacyTheorem} simplifies establishing the lower bounds on the propositional modal validities of a modal system. However, determining upper bounds presents a different challenge. To address this, I develop the control statement theory~\cite{HamkinsLoewe2008:TheModalLogicOfForcing, HamkinsLeibmanLoewe2015:StructuralConnectionsForcingClassAndItsModalLogic, HamkinsLinnebo:Modal-logic-of-set-theoretic-potentialism, HamkinsWilliams:The-universal-finite-sequence, HamkinsWoodin:The-universal-finite-set, Hamkins:The-modal-logic-of-arithmetic-potentialism, HamkinsWoloszyn:Modal-model-theory} for the modal logics \theoryf{Grz}, \theoryf{Grz.2}, and \theoryf{Grz.3}. This work resolves an open question regarding the upper bounds in the absence of switches among the control statements used for \theoryf{S4}, \theoryf{S4.2}, and \theoryf{S4.3}~\cite{Passmann19-are-buttons-really-enough-to-bounds-validities-by-S4.2}.

\newtheorem*{UpperBoundsTheorem}{Upper Bounds Theorem}%
\begin{UpperBoundsTheorem}
\makeatletter\def\@currentlabelname{upper bounds theorem}\makeatother\label{UpperBoundsTheorem}
    Suppose $w$ is an initial world in a Kripke model $M$.
    \begin{enumerate}
        \item The propositional modal assertions valid at $w$ are contained in the modal logic \theoryf{Grz.3} if and only if $w$ admits ratchets of arbitrarily large finite lengths.
        \item The propositional modal assertions valid at $w$ are contained in the modal logic \theoryf{Grz.2} if and only if $w$ admits arbitrarily large finite families of independent buttons.
        \item The propositional modal assertions valid at $w$ are contained in the modal logic \theoryf{Grz} if and only if $w$ admits a labeling for every regular finite tree $T$.
    \end{enumerate}
\end{UpperBoundsTheorem}

It has long been known that modal logic \theoryf{Grz.2} corresponds to the class of Kripke frames consisting of all antiwellfounded directed partial orders~\cite{Kracht1999:Tools-and-Techniques-in-Modal-Logic}. By employing the methods of~\cite{Fine1974:Logics-containing-K4-Part-I, Fine1985:Logics-containing-K4-Part-II}, one can show that finite directed partial orders suffice for the characterization. Using the new technology from the~\nameref{UpperBoundsTheorem} and drawing upon the methods presented in~\cite{HamkinsLoewe2008:TheModalLogicOfForcing}, I give a further refinement and show that finite Boolean algebras are enough.

\newtheorem*{CharacterizationTheorem}{Characterization Theorem}%
\begin{CharacterizationTheorem}
\makeatletter\def\@currentlabelname{characterization theorem}\makeatother\label{CharacterizationTheorem}
    The modal logic \theoryf{Grz.2} is characterized by finite Boolean algebras. That is, a propositional modal assertion $\varphi$ is provable in $\theoryf{Grz.2}$ if and only if $\varphi$ is valid in every finite Boolean algebra. 
\end{CharacterizationTheorem}

The~\nameref{CharacterizationTheorem} is the best possible, meaning that no proper subclass of the class of all finite Boolean algebras characterizes \theoryf{Grz.2}.

\section{Demystifying Grzegorczyk logic}

In this section, I should like to explain what Grzegorczyk logic is about by replacing the \axiomf{Grz} axiom with a new variant that is easier to understand. I am greatly indebted to Emil Jeřábek for his comments and suggestions that helped me finish this section.

Before I state the axiom, let me introduce some preparatory notation and terminology. We will say that $p$ is \emph{penultimate}, and denote it by $\mathrm{penultimate}(p)$, when $p \land \possible \neg p \land \necessary(\neg p \implies \necessary \neg p)$. This condition says that $p$ is true but possibly false, and once it becomes false it remains necessarily false. Moreover, for convenience, let us use the notation $\mathrm{contingent}(p)$ to denote the contingency of $p$. That is $\possible p \land \possible \neg p$. I am now poised to introduce the axiom.
$$
\begin{array}{rl}
    \axiomf{Grz}^* & \mathrm{contingent}(p) \implies \possible (\mathrm{penultimate}(p) \lor \mathrm{penultimate}(\neg p)).
\end{array}
$$

Let me denote the normal modal logic generated by \(\axiomf{Grz}^*\) as \(\theoryf{Grz}^*\). I shall demonstrate that this logic is equivalent to Grzegorczyk logic over \(\theoryf{S4}\), thereby proving the \nameref{PenultimacyTheorem}. Theorem~\ref{Theorem.my-Grz} is simply a technical restatement of the \nameref{PenultimacyTheorem}.

\medskip

We say that $p$ \emph{attains} a penultimate truth-value if $p$ or $\neg p$ is penultimate.

\begin{theorem}\label{Theorem.my-Grz}
    The modal logic $\theoryf{Grz}^{*}$ is equivalent to Grzegorczyk logic over $\theoryf{S4}$. Consequently, $\theoryf{Grz}$ is the smallest reflexive and transitive normal modal logic in which every contingency possibly attains a penultimate truth-value. 
\end{theorem}

We divide the proof into two parts: lemma~\ref{Lemma.Over-S4-Grz*-contains-Grz} assumes $\theoryf{S4}$ and shows that $\axiomf{Grz}^{*}$ implies $\axiomf{Grz}$, whereas lemma~\ref{Lemma.Grz-contains-Grz*} shows that $\theoryf{Grz}\vdash \axiomf{Grz}^{*}$, in any case.

\medskip

It is convenient to reformulate the Grzegorczyk axiom as $p \lor \possible (\neg p \land \necessary(p \implies \necessary p))$. We achieve this by simply rewriting implications as disjuncts and applying De Morgan's laws for modal logic.

Let me call $p$ \emph{weakly penultimate}, and denote it by $\mathrm{penultimate}^{-}(p)$, when $p \land \necessary(\neg p \implies \necessary \neg p)$. This condition says that $p$ is true but if it ever becomes false, it actually becomes necessarily false. The key difference in this weakening is that $p$ can be necessary. The axiom then simplifies to the following concise form.
    $$
    \begin{array}{rl}
        \axiomf{Grz} & p \lor \possible \mathrm{penultimate}^{-}(\neg p).
    \end{array}
    $$

\begin{lemma}\label{Lemma.Over-S4-Grz*-contains-Grz}
    Over the modal logic \theoryf{S4}, the modal axiom $\axiomf{Grz}^{*}$ directly implies the modal axiom \axiomf{Grz}.
\end{lemma}

\begin{proof}
    The easy cases are when $p$ is either true or necessarily false. The former is immediate. If the latter holds, then $\necessary(p \implies \necessary p)$ holds trivially, yielding $\neg p \land \necessary(p \implies \necessary p)$. In $\theoryf{S4}$, what is true is possible, thus $\possible(\neg p \land \necessary(p \implies \necessary p))$ as required.

    Suppose $p$ is false but possibly true. Then, the axiom $\axiomf{Grz}^{*}$ implies that either $p$ or $\neg p$ is possibly penultimate. Should it be the latter, we are done---penultimacy implies weak penultimacy. Otherwise, there is a possible world where $\mathrm{penultimate}(p)$ holds. In that world, $\possible \neg p$ holds, so there is a further accessible world where $\neg p$ holds. Since $\necessary(\neg p \implies \necessary \neg p)$ held at the intermediate world, this further world satisfies $\necessary \neg p$, and hence also $\necessary(p \implies \necessary p)$. Therefore $\mathrm{penultimate}^{-}(\neg p)$ is true at that further world. By transitivity, $\possible \mathrm{penultimate}^{-}(\neg p)$ holds in the original world.
\end{proof}

Let me discuss the Kripke frame theory for the modal logic \(\theoryf{Grz}\), as there appears to be a misunderstanding in the literature. Kripke frames validating \(\theoryf{Grz}\) have no proper clusters, meaning there are no distinct mutually accessible worlds. The logics corresponding to such classes of frames are said to have \emph{fatness} one. However, it is misguided to interpret this as indicating that \(\theoryf{Grz}\) corresponds to the class of all reflexive and transitive Kripke frames of fatness one. 

Consider the counterexample illustrated in figure~\ref{Figure.Grz-counterexample-.2-valid}, which constitutes a reflexive and transitive Kripke frame of fatness one that does not validate \(\theoryf{Grz}\).

\begin{figure}[h]
    $$
\begin{tikzcd}
w_0 \arrow[r] & w_1 \arrow[r] & w_2 \arrow[r] & \cdots
\end{tikzcd}
$$
    \caption{The infinite $\omega$-chain (with reflexive closure): a fatness-one reflexive transitive frame that fails to validate $\theoryf{Grz}$.}
    \label{Figure.Grz-counterexample-.2-valid}
\end{figure}

To see why \(\theoryf{Grz}\) is not valid in this frame, consider a valuation for which \(w_i(p)\) is true if and only if \(i\) is odd. Then every world \(w_i\) satisfies \(\possible p \land \possible \neg p\), and hence \(w_0 \models \necessary(\possible p \land \possible \neg p)\). Consequently \(\necessary(p\implies\necessary p)\) fails at every world, so \(w_0 \models \necessary(\necessary(p\implies\necessary p)\implies p)\) while \(w_0 \models \neg p\). Thus the axiom \(\axiomf{Grz}\) fails at \(w_0\).

\medskip

Let me say that a Kripke frame exhibits \emph{subframe stabilization} if for every non-empty subframe $F$, there exists a node $x$ in $F$ such that for all $y \geq_F x$, the interval $[x, y]$ is contained in $F$. Emil Jeřábek proved in~\cite{Jerabek2004-Grzegorczyk} that the modal logic $\theoryf{Grz}$ corresponds to the class of reflexive and transitive Kripke frames that admit subframe stabilization. He also proved that under the axiom of dependent choice, this condition is equivalent to the frame being \emph{antiwellfounded}. In set-theoretic literature, this property is referred to as \emph{upwards} or \emph{converse wellfoundedness}. In the field of algebra, it is known as \emph{Noetherianity}. Whatever the name, it simply means that the frame is wellfounded if we reverse the direction of its arrows.

Perhaps the confusion I alluded to earlier stems from the fact that the modal logic \(\theoryf{Grz}\) is \emph{characterized} by the class of all finite partial orders (equivalently, all finite reflexive and transitive Kripke frames of fatness one). Indeed, Grzegorczyk logic has the \emph{finite frame property}, which means that it is characterized by a class of finite frames, or, equivalently, every non-theorem fails on some finite frame. In \ZFC, the logic \(\theoryf{Grz}\) corresponds to the class of all antiwellfounded frames. So by the finite frame property, it is characterized by the class of all finite antiwellfounded frames. These are just partial orders, and therefore amount to finite reflexive and transitive Kripke frames.

To see that the same holds in \(\mathsf{ZF}\), one can either directly show that finite Kripke frames with frame stabilization are exactly finite partial orders. Or, one can observe that the modal logic \(\theoryf{Grz}\) as well as the set of all partial orders are both recursively enumerable, and so are absolute to \(\mathrm L\) by Shoenfield's absoluteness theorem.

\begin{lemma}\label{Lemma.Grz-contains-Grz*}
    The modal logic \theoryf{Grz} contains the modal logic $\theoryf{Grz}^{*}$.
\end{lemma}

In order to highlight various aspects of the modal logic \theoryf{Grz}, I shall present two proofs of
this lemma. The first proof is elementary, while the second is more technical, detailing the precise derivation rules and substitution instances that can be used to obtain $\theoryf{Grz}^{*}$ from $\theoryf{Grz}$.

\begin{proof}[Proof one]
    Observe that \theoryf{Grz} is valid in every Kripke frame that is a finite partial order. It therefore suffices to justify that the axiom $\axiomf{Grz}^*$ is also valid in every finite partial order.

    So let $F$ be a finite partial order, let $M$ be a Kripke model based on $F$, and let $w$ be a world in $M$. Assume $(M,w)\satisfies \mathrm{contingent}(p)$, that is $(M,w)\satisfies \possible p\land \possible\neg p$. Consider the set $C=\set{u\in F\mid w\leq_F u\ \text{and}\ (M,u)\satisfies \mathrm{contingent}(p)}$. This set is non-empty, as it contains $w$, and finite, hence it has a $\leq_F$-maximal element $v\in C$. We claim that $(M,v)\satisfies \mathrm{penultimate}(p)\ \lor\ \mathrm{penultimate}(\neg p)$.

    Indeed, either $(M,v)\satisfies p$ or $(M,v)\satisfies \neg p$. Suppose $(M,v)\satisfies p$ (the other case is symmetric).
    Since $v\in C$, we have $(M,v)\satisfies \possible\neg p$. It remains to show that $(M,v)\satisfies \necessary(\neg p\implies \necessary\neg p)$.
    So let $u\geq_F v$ and assume $(M,u)\satisfies \neg p$. Because $F$ is reflexive, $(M,u)\satisfies \possible\neg p$. If also $(M,u)\satisfies \possible p$, then $(M,u)\satisfies \mathrm{contingent}(p)$, and hence $u\in C$.
    But then $u\geq_F v$ with $u\neq v$ contradicts the maximality of $v$ in $C$. Therefore $(M,u)\not\satisfies \possible p$, that is, $(M,u)\satisfies \necessary\neg p$.
    This proves $\necessary(\neg p\implies \necessary\neg p)$ at $v$, and so $(M,v)\satisfies \mathrm{penultimate}(p)$.

    Consequently,
    \begin{equation*}
        (M,w)\satisfies \possible\big(\mathrm{penultimate}(p)\lor \mathrm{penultimate}(\neg p)\big),
    \end{equation*}
    so $F$ validates $\axiomf{Grz}^*$. Since $F$ was an arbitrary finite partial order, $\axiomf{Grz}^*$ is valid in every finite partial order.

    Finally, the modal logic \theoryf{Grz} has the finite frame property and is characterized by the class of all finite partial orders. Therefore, by completeness of Grzegorczyk logic, $\theoryf{Grz}\vdash \axiomf{Grz}^*$, and hence the modal logic generated by $\axiomf{Grz}^*$ is contained in \theoryf{Grz}, that is, $\theoryf{Grz}^*\subseteq \theoryf{Grz}$.
\end{proof}

The second proof utilizes the following lemma.
\begin{technicallemma*}
    \begin{equation*}
        \theoryf{K} \proves \necessary((\possible p \implies p) \to \necessary(\possible p \implies p)) \to \necessary(p \implies \necessary p) \lor \possible \mathrm{penultimate}(p).
    \end{equation*}
\end{technicallemma*}

\begin{proof}[Proof]
    Assume $\necessary\big((\possible p \implies p)\to \necessary(\possible p \implies p)\big)$ and suppose $\neg\necessary(p \implies \necessary p)$. Then there is an accessible world $v$ such that $p$ holds at $v$ and $\neg\necessary p$ holds at $v$, hence $\possible\neg p$ holds at $v$. Since $p$ holds at $v$, the implication $(\possible p \implies p)$ holds at $v$ trivially, and therefore $v$ satisfies $\necessary(\possible p \implies p)$ by the assumed boxed implication. Using the equivalence $(\possible p \implies p) \leftrightarrow (\neg p \implies \necessary\neg p)$, we obtain that $v$ satisfies $\necessary(\neg p \implies \necessary\neg p)$. Altogether, $v$ satisfies $p \land \possible\neg p \land \necessary(\neg p \implies \necessary\neg p)$, that is, $\mathrm{penultimate}(p)$. Hence $\possible \mathrm{penultimate}(p)$ holds at the original world.
\end{proof}

\begin{proof}[Proof two]
    First, let us substitute $\possible p \implies p$ into $\axiomf{Grz}$ (in the form
    $q \lor \possible \mathrm{penultimate}^{-}(\neg q)$). This substitution gives $(\possible p \implies p) \lor \possible \mathrm{penultimate}^{-}(\neg (\possible p \implies p))$. Expanding $\mathrm{penultimate}^{-}$, we have $\mathrm{penultimate}^{-}(\neg (\possible p \implies p))
        \leftrightarrow
        \neg (\possible p \implies p) \land \necessary\big((\possible p \implies p)\to \necessary(\possible p \implies p)\big)$. By the technical lemma, $\necessary\big((\possible p \implies p)\to \necessary(\possible p \implies p)\big)
        \to
        \necessary(p \implies \necessary p) \lor \possible \mathrm{penultimate}(p)$.
    Hence, $\mathrm{penultimate}^{-}(\neg (\possible p \implies p))
        \to
        \big(\neg(\possible p \implies p)\land \necessary(p \implies \necessary p)\big)
        \lor
        \big(\neg(\possible p \implies p)\land \possible \mathrm{penultimate}(p)\big)$.
    Using $\neg(\possible p \implies p)\leftrightarrow (\possible p \land \neg p)$, the first disjunct becomes $
        \neg p \land \possible p \land \necessary(p \implies \necessary p)$,
    which is exactly $\mathrm{penultimate}(\neg p)$. The second disjunct implies $\possible \mathrm{penultimate}(p)$. Thus, $
        \mathrm{penultimate}^{-}(\neg (\possible p \implies p))
        \to
        \mathrm{penultimate}(\neg p)\ \lor\ \possible \mathrm{penultimate}(p).
    $

    By monotonicity of $\possible$, from the substitution $\possible p \implies p$ we obtain
    \begin{equation*}
            (\possible p \implies p) \lor \possible\big(\mathrm{penultimate}(\neg p) \lor \possible \mathrm{penultimate}(p)\big).
    \end{equation*}
    Rewriting as an implication, we get
    \begin{equation*}
    \neg(\possible p \implies p)
        \to
        \possible\big(\mathrm{penultimate}(\neg p) \lor \possible \mathrm{penultimate}(p)\big).
    \end{equation*}
    Since $\theoryf{K4}\subseteq \theoryf{Grz}$, transitivity yields the valid implication
    \begin{equation*}
        \possible(\alpha \lor \possible\beta)\to \possible(\alpha \lor \beta).
    \end{equation*}
    Applying this with $\alpha=\mathrm{penultimate}(\neg p)$ and $\beta=\mathrm{penultimate}(p)$, we obtain
    \begin{equation*}
        \neg(\possible p \implies p)
        \to
        \possible\big(\mathrm{penultimate}(\neg p) \lor \mathrm{penultimate}(p)\big).
    \end{equation*}
     Similarly, substituting $\possible \neg p \implies \neg p$ into $\axiomf{Grz}$ yields
   \begin{equation*}
        \neg(\possible \neg p \implies \neg p)
        \to
        \possible\big(\mathrm{penultimate}(\neg p) \lor \mathrm{penultimate}(p)\big).
    \end{equation*}

    Finally, $\mathrm{contingent}(p)$ implies $\neg(\possible p \implies p) \lor \neg(\possible \neg p \implies \neg p)$,
    since either $p$ or $\neg p$ holds at the current world. As each disjunct implies
    $\possible(\mathrm{penultimate}(\neg p)\lor \mathrm{penultimate}(p))$, we conclude
    \begin{equation*}
        \mathrm{contingent}(p)
        \to
        \possible\big(\mathrm{penultimate}(\neg p)\ \lor\ \mathrm{penultimate}(p)\big),
    \end{equation*}
    which is exactly $\axiomf{Grz}^*$. Thus $\theoryf{Grz}\vdash \axiomf{Grz}^*$, that is,  $\theoryf{Grz}^{*}\subseteq \theoryf{Grz}$.
\end{proof}

I should like to emphasize that the modal logics $\theoryf{Grz}^*$ and $\theoryf{Grz}$ are not equivalent, except under the additional assumption that $\theoryf{S4}$ is valid. The distinction is evident in the following observation.

\begin{observation}\label{Observation.Grz-start-not-equivalent-to-Grz}
    The modal logics $\theoryf{Grz}^*$ and \theoryf{Grz} are not equivalent in general. More precisely, $\theoryf{Grz}^* \ofneq \theoryf{Grz}$ over \theoryf{K}.
\end{observation}

\begin{proof}
    Recall that $\mathrm{contingent}(p)$ stands for $\possible p \land \possible \neg p$. We can rewrite the implication in the axiom $\axiomf{Grz}^*$ as a disjunction to obtain its equivalent form:
    $$
    \begin{array}{ll}
        \axiomf{Grz}^* & \necessary p \lor \necessary \neg p \lor \possible (\mathrm{penultimate}(p) \lor \mathrm{penultimate}(\neg p)).
    \end{array}
    $$
    Consider the axiom of bounded alternativity one, below.
    $$
    \begin{array}{ll}
        \axiomf{Alt}_1 & \necessary p \lor \necessary \neg p \\
    \end{array}
    $$

    By comparing the axioms, it is evident that the logic generated by $\axiomf{Alt}_1$ contains $\axiomf{Grz}^*$. This is the modal logic $\theoryf{Alt}_1$, corresponding to the class of all Kripke frames in which each node has at most one successor. It follows that $\theoryf{Grz}^*$ must be valid in all frames from that class.

    On the other hand, observe that such frames need not validate the modal logic $\theoryf{S4}$, which is included in \theoryf{Grz}. And so, there exists a Kripke frame in which $\theoryf{Grz}^*$ is valid but $\theoryf{Grz}$ is not. Consequently, the two logics must not be equivalent.
\end{proof}

In observation~\ref{Observation.Grz-start-not-equivalent-to-Grz}, we have established that the modal logic $\theoryf{Grz}^*$ is valid on a broader class of frames than \theoryf{Grz}. This raises the following question.

\begin{openquestion}\label{Question.Grz-start-correspondence}
    What is the Kripke frame correspondence for the modal logic $\theoryf{Grz}^*$?
\end{openquestion}

Situations where the modal logic \(\theoryf{S4}\) is not valid, such as the one in question~\ref{Question.Grz-start-correspondence}, are beyond the scope of this article. Therefore, I leave the investigation of such cases to the curious reader. 

\bigskip

To streamline our discussion, we will henceforth assume that all referenced Kripke frames and models validate \(\theoryf{S4}\). Consequently, all Kripke frames we consider will be reflexive and transitive.

\section{\theoryf{Grz.2} is characterized by finite lattices}\label{Section.Grz.2-complete-for-finite-lattices}

In this section, I present the first significant refinement of the finite frame characterization of the modal logic \theoryf{Grz.2}. We will improve this even further in section~\ref{Section.Grz.2-complete-for-finite-Boolean-algebras}. But this will require the control statement theory for $\theoryf{Grz.2}$, which we will develop in section~\ref{Section.Control-statement-theory}. The current section lays the foundations for this endeavor.

\begin{theorem}\label{Theorem.Grz.2-is-characterized-by-finite-lattices}
    The modal logic \theoryf{Grz.2} is characterized by the class of all finite lattices. That is, a propositional modal assertion $\varphi$ is provable in $\theoryf{Grz.2}$ if and only if $\varphi$ is valid in every Kripke frame that is a finite lattice.
\end{theorem}

It is immediate to show that \theoryf{Grz.2} is sound with respect to the class of all finite lattices. Indeed, every assertion provable in \theoryf{Grz.2} must be valid in every finite directed partial order, and hence every finite lattice in particular. It therefore suffices to show that \theoryf{Grz.2} is complete with respect to the set of all finite lattices.

To be complete with respect to a class of frames is a less severe requirement than to be complete with respect to its proper subclass. And so, in order to show the completeness of \theoryf{Grz.2} with respect to lattices, we will demonstrate the completeness for the class of all baled trees.

\medskip

A \emph{baled tree} is a partially ordered set $T$ that has the greatest element $t \in T$, with the property that when this element is removed, the remaining set $T \setminus \singleton{t}$ is a tree. The idea behind a baled tree is to imagine an upward-growing tree with its topmost part gathered together and tied into a bale. The term was coined by Joel David Hamkins and Benedikt {\Lowe} in~\cite{HamkinsLoewe2008:TheModalLogicOfForcing}.

\medskip

The following lemma completes the proof of theorem~\ref{Theorem.Grz.2-is-characterized-by-finite-lattices}.

\begin{lemma}\label{Lemma.not-part-of-Grz.2-fails-in-baled-tree}
A modal assertion that is not provable in the modal logic \theoryf{Grz.2} has to fail in some Kripke model whose Kripke frame is a finite baled tree, and hence a finite lattice.
\end{lemma}

\begin{proof}
    We follow the ideas developed in~\cite[Lemma 6.5]{HamkinsLoewe2008:TheModalLogicOfForcing}. The logic \theoryf{Grz.2} is characterized by the class of all finite directed partial orders. Therefore, if \theoryf{Grz.2} does not prove an assertion $\varphi$, then there is a Kripke model $N$ based on a finite directed partially ordered Kripke frame $F$, together with a world $u_0$ in it at which $\varphi$ fails. We will construct a Kripke model $M$ that is a finite baled tree and bisimilar to $N$. We use the technique of \emph{partial tree unraveling}, cf. the proof of~\cite[Lemma 6.5]{HamkinsLoewe2008:TheModalLogicOfForcing}.

    We say that $t$ is a \emph{path} from $u_0$ to $u$ in $F$ if and only if it is a maximal chain in the interval $[u_0, u]$.
    Let
    \begin{equation*}
        T = \set{t \mid \text{there is some $u$ such that $t$ is a path from $u_0$ to $u$}}
    \end{equation*}
    endowed with a partial order $\leq_T$ where $t_0 \leq_T t_1$ if and only if $t_1$ is an end-extension of $t_0$, so that $T$ forms a tree.
    Let $B$ be the partial tree unraveling of $F$.
    Specifically, since $F$ is finite and directed it has a greatest element $g$.
    Let $B$ consist of $g$ together with all pairs $\langle u,t\rangle$ where $t$ is a path from $u_0$ to $u$.
    Define $\langle u,t\rangle \leq_B \langle u',t'\rangle$ if and only if $u \leq_F u'$ and $t \leq_T t'$, and additionally stipulate that $\langle u,t\rangle \leq_B g$ for every pair $\langle u,t\rangle\in B$.
    Observe that $B$ forms a baled tree, with $g$ being its greatest element.

    We now construct the desired Kripke model $M$ based on $B$. Namely, we copy the values of propositional variables from each world $u \in N$ to all its copies $\<u,t>$ in $B$, and we let the distinguished node $g$ of $B$ inherit the valuation of $g$ in $N$. Each world $\<u,t>$ in $M$ accesses exactly those copies $\<u',t'>$ corresponding to accessible worlds $u'\geq_F u$ whose witnessing paths $t'$ end-extend $t$; in particular, this yields a bisimulation between $N$ and $M$ via the natural ``copy'' relation. As a result, every world in $M$ has exactly the same modal truths as its copies in $N$.
    
    Consequently, $\varphi$ fails in $M$ at a world that is a copy of $u_0$. Thus, $\varphi$ fails in a Kripke model with a Kripke frame that is a finite baled tree, and hence a lattice.
\end{proof}

The results of this section are summarized by the following theorem. We will expand on it in section~\ref{Section.Grz.2-complete-for-finite-Boolean-algebras}.

\begin{theorem}\label{Theorem.Grz-2-characterization-partial}
    The following sets of finite Kripke frames characterize the modal logic $\theoryf{Grz.2}$.
    \begin{enumerate}
        \item Finite directed partial orders.
        \item Finite lattices.
        \item Finite baled trees.
    \end{enumerate}
\end{theorem}

\section{Model and frame labelings}

The constructions of this section lie at the heart of the control statement theory, serving as the primary technique of simulating modal truth of one Kripke model in the other.

\medskip

Let us start by introducing a new notion of Kripke model labeling. Suppose $M$ and $N$ are Kripke models with initial worlds $w_0$ and $u_0$. A \emph{model labeling} of $(M,w_0)$ for $(N,u_0)$ is a truth-preserving assignment of an assertion $\psi_p$ to each propositional variable $p$ in the propositional modal language. In other words, there exists an assignment $p\mapsto\psi_p$ such that for any modal assertion $\varphi$, we have
    \begin{equation*}
        (M,w_0) \satisfies \varphi(p_0,\ldots,p_k) \quad \text{if and only if} \quad (N,u_0) \satisfies \varphi(\psi_{p_0},\ldots,\psi_{p_k}).
    \end{equation*}

The point of model labeling is that whenever an assertion $\varphi$ is false at $(M,w_0)$, a suitable substitution instance of $\varphi$ is false at $(N,u_0)$. Consequently, the propositional modal validities of the latter are contained in those of the former. Throughout this article we work with pointed frames and models. When analyzing the modal assertions true (or valid) at a world $w$, we may replace the ambient frame or model by the subframe or submodel generated by $w$, and hence treat $w$ as an initial world without loss of generality. The theory of a pointed model $(N,u_0)$ is the set of modal assertions true at $u_0$ in $N$.

\newtheorem*{ModelFrameLabelingLemma}{Model Labeling Lemma}%
\begin{ModelFrameLabelingLemma}
\makeatletter\def\@currentlabelname{model labeling lemma}\makeatother\label{ModelFrameLabelingLemma}
Suppose $F$ is a Kripke frame with an initial node $w_0$ and $N$ a Kripke model with an initial world $u_0$.
\begin{enumerate}
  \item If for every Kripke model $M$ based on $F$ there is a model labeling of $(M,w_0)$ for $(N,u_0)$, then the modal logic of the underlying frame of $N$ is contained in the modal logic of $F$.
  \item If there is a Kripke model $M$ based on $F$ such that there is a model labeling of $(N,u_0)$ for $(M,w_0)$, then the modal logic of $F$ is contained in the theory of $(N,u_0)$.
\end{enumerate}
\end{ModelFrameLabelingLemma}

\begin{proof}
    Let us tackle statement (1) first, arguing by contrapositive. If $\varphi$ is not valid in $F$, then there exists a Kripke model $M$ based on $F$ and an initial world $w_0$ with $(M,w_0)\not\satisfies \varphi(p_0,\ldots,p_k)$. By assumption, there is a model labeling $p \mapsto \psi_p$ of $(M,w_0)$ for $(N,u_0)$, whence $(N,u_0)\not\satisfies \varphi(\psi_{p_0},\ldots,\psi_{p_k})$. Thus, some substitution instance of $\varphi$ fails at $(N,u_0)$. Since the logic of the underlying frame of $N$ is closed under uniform substitution, it follows that $\varphi$ is not valid on that frame. Hence the modal logic of the frame of $N$ is contained in the modal logic of $F$.

    For statement (2), there is a Kripke model $M$ based on $F$ and a labeling $p \mapsto \psi_p$ of $(N,u_0)$ for $(M,w_0)$ such that for every modal assertion $\varphi$,
    \begin{equation*}
        (N,u_0) \satisfies \varphi(p_0,\ldots,p_k) \quad \text{if and only if} \quad (M,w_0) \satisfies \varphi(\psi_{p_0},\ldots,\psi_{p_k}).
    \end{equation*}
    Suppose $\varphi$ is valid in $F$. By closure of frame validity under uniform substitution, $\varphi(\psi_{p_0},\ldots,\psi_{p_k})$ is valid in $F$, hence $(M,w_0)\satisfies \varphi(\psi_{p_0},\ldots,\psi_{p_k})$. By the displayed equivalence, $(N,u_0)\satisfies \varphi$. Thus the modal logic of $F$ is contained in the set of assertions true at $(N,u_0)$, that is the theory of $(N,u_0)$.
\end{proof}
    
A \emph{frame labeling} of a Kripke model, on the other hand, involves labeling its underlying Kripke frame. Specifically, a labeling of a Kripke frame $F$ for $(N,u_0)$, where $N$ is a Kripke model and $u_0$ is an initial world in it, is an assignment to each node $w$ in $F$ an assertion $\Phi_w$ in the propositional modal language, such that
\begin{enumerate}
    \item $(N,u_0) \satisfies \Phi_{w_0}$, where $w_0$ is a given initial node of $F$.
    \item For every world $u$ accessible from $u_0$, if $(N,u) \satisfies \Phi_w$, then
    $(N,u) \satisfies \possible \Phi_{w'}$ if and only if $w \leq_F w'$.
    \item For every world $u$ accessible from $u_0$, there is exactly one $\Phi_w$ such that $(N,u) \satisfies \Phi_w$.
\end{enumerate}

If a Kripke frame $F$ is finite, the conditions (1)--(3) are expressible in the language of propositional modal logic by the \emph{Jankov-Fine} formula for Kripke frame $F$:
\begin{equation*}
    p_{w_0} \land \necessary \big( \bigvee_{w \in F} p_w \land \bigwedge_{w \neq v} (p_w \implies \neg p_v) \land \bigwedge_{w \leq v} (p_w \implies \possible p_v) \land \bigwedge_{w \not\leq v} (p_w \implies \neg \possible p_v) \big).
\end{equation*}

Jankov-Fine formulas were used in~\cite{HamkinsLoewe2008:TheModalLogicOfForcing} to establish the upper bounds on the modal logic of forcing. And~\cite{HamkinsLeibmanLoewe2015:StructuralConnectionsForcingClassAndItsModalLogic} abstracted the idea out with the notion of a frame labeling.

\medskip

Next, let me demonstrate that from a labeling of a finite frame, we can derive a model labeling for every model based on that frame.

\newtheorem*{FrameLabelingLemma}{Frame Labeling Lemma}%
\begin{FrameLabelingLemma}
\makeatletter\def\@currentlabelname{frame labeling lemma}\makeatother\label{FrameLabelingLemma}
    Suppose $F$ is a finite Kripke frame with initial node $w_0$, and $(N,u_0)$ is a pointed Kripke model. If there is a frame labeling of $F$ for $(N,u_0)$, then for every Kripke model $M$ based on $F$ there exists a model labeling of $(M,w_0)$ for $(N,u_0)$. Consequently, the modal logic of the underlying frame of $N$ is contained in the modal logic of $(F,w_0)$.
\end{FrameLabelingLemma}

\begin{proof}
    The proof closely follows~\cite{HamkinsLeibmanLoewe2015:StructuralConnectionsForcingClassAndItsModalLogic}, although that work was exclusively about models of set theory under forcing interpretation of modality.

    Let $w \mapsto \Phi_w$ be a labeling of a Kripke frame $F$ for $(N,u_0)$, where $N$ is a Kripke model and $u_0$ an initial world in it. Let $M$ be any Kripke model based on $F$. We recover the model labeling $p \mapsto \psi_p$ from the frame labeling by letting $\psi_p$ be the disjunction of all statements $\Phi_w$ associated with each node $w$ in $F$ whose valuation of $p$ at $M$ is true. In other words, we view each $w \in F$ as a propositional world in $M$ and let
    \begin{equation*}
        \psi_p = \bigvee \set{\Phi_w \mid (M,w) \satisfies p}.
    \end{equation*}
    To see that $p \mapsto \psi_p$ is indeed a model labeling, it suffices to show simultaneously for all worlds $w$ and all worlds $u$ accessible from $u_0$, by induction on the complexity of formulas, that if $(N,u) \satisfies \Phi_w$, then
    \begin{equation*}
        (M,w) \satisfies \varphi(p_0,\ldots,p_k) \quad \text{if and only if} \quad (N,u) \satisfies \varphi(\psi_{p_0},\ldots,\psi_{p_k}).
    \end{equation*}
    This is the uniform version of the model labeling. It is clear that the claim follows from it.

    The atomic case is immediate from the definition of $\psi_p$. The case of Boolean combinations is routine. Let us consider the modal operator case. Look at the forward direction first. Suppose $(M, w) \satisfies \possible \varphi$, so there is a world $w^+ \geq_F w$ such that $(M, w^+) \satisfies \varphi$. Let $u$ be any world in $N$ accessible from $u_0$ such that $(N,u)\satisfies \Phi_w$. Since $w \leq_F w^+$, by the definition of frame labeling we have $(N,u)\satisfies \possible \Phi_{w^+}$, so there is a world $u^+$ accessible from $u$ with $(N,u^+)\satisfies \Phi_{w^+}$. By the induction hypothesis, $(N,u^+)\satisfies \varphi(\psi_{p_0},\ldots,\psi_{p_k})$. Hence $(N,u)\satisfies \possible \varphi(\psi_{p_0},\ldots,\psi_{p_k})$, as desired.

    For the converse direction, suppose $(N,u) \satisfies \possible \varphi(\psi_{p_0},\ldots,\psi_{p_k})$. Then there is a world $u^+$ in $N$ accessible from $u$ such that $(N,u^+) \satisfies \varphi(\psi_{p_0},\ldots,\psi_{p_k})$. By the definition of frame labeling, $u^+$ has a corresponding label that it satisfies, say $(N,u^{+}) \satisfies \Phi_{w^+}$. From this and the induction hypothesis, we get that $(M,w^{+}) \satisfies \varphi$. Since $u$ accesses $u^+$, we have $(N,u)\satisfies \possible \Phi_{w^+}$, and so, by the definition of the frame labeling again, we must have that $w \leq_F w^+$. Thus, $(M,w) \satisfies \possible \varphi$, and we are done.
\end{proof}

Let me end this section with a digression. In~\cite{HamkinsLeibmanLoewe2015:StructuralConnectionsForcingClassAndItsModalLogic}, the authors suggested that Jankov-Fine formulas and frame labelings are equivalent. In light of the hypothesis of the \nameref{FrameLabelingLemma}, one might argue that my emphasis on it being the case only when the frame is finite is overly particular. 

However, let me address this by suggesting a template for variants of the \nameref{FrameLabelingLemma} that would work for infinite Kripke frames. The finiteness of the Kripke frame guarantees that the disjunction is finite, and thus \(\psi_p\) is an assertion in the language of propositional modal logic. Therefore, it is the property of the language in which we express the labels that necessitates that the frame be finite, rather than the Kripke frame itself, the Kripke models based on that frame, or the underlying language of these models.

Therefore, if the language we use to label allowed for infinitary disjunctions of the same size as the Kripke frame, then \(\psi_p\) would be expressible in that language, and the argument goes through.

\section{Control statement theory for Grzegorczyk logics}\label{Section.Control-statement-theory}

Control statement theory has been an enormously successful methodology in establishing the upper bounds on the propositional modal validities of a modal system. It emerged in a series of papers on the modal logics of set theory and arithmetic~\cite{HamkinsLoewe2008:TheModalLogicOfForcing},\cite{HamkinsLeibmanLoewe2015:StructuralConnectionsForcingClassAndItsModalLogic}, \cite{HamkinsLinnebo:Modal-logic-of-set-theoretic-potentialism}, \cite{HamkinsWilliams:The-universal-finite-sequence}, \cite{HamkinsWoodin:The-universal-finite-set}, \cite{Hamkins:The-modal-logic-of-arithmetic-potentialism}.

With this method, one uses the existence of
various kinds of control statements in the modal system---buttons, switches,
dials, ratchets, or tree labelings---to establish upper bounds on the class of propositional modal validities of the system. Together with the results of this article, we can summarize the state-of-the-art as follows.

\begin{table}[h]
\begin{adjustbox}{center}
\begin{tabular}{@{}ll@{}}
\toprule
\multicolumn{1}{c}{Upper bounds}  & \multicolumn{1}{c}{Control statements} \\
\cmidrule(l){1-2}
\multicolumn{1}{c}{Modal logic} & \multicolumn{1}{c}{Arbitrarily large finite families of independent} \\
\cmidrule(l){1-2}
\multicolumn{1}{c}{S4.3} & ratchets and switches \\
\multicolumn{1}{c}{S4.2} & buttons and switches \\
\multicolumn{1}{c}{Grz.3} & ratchets \\
\multicolumn{1}{c}{Grz.2} & buttons \\
\cmidrule(l){2-2}
& \multicolumn{1}{c}{Labelings of every finite regular} \\
\cmidrule(l){2-2}
\multicolumn{1}{c}{S4}   & pre-tree \\
\multicolumn{1}{c}{Grz}   & tree \\
\bottomrule
\end{tabular}
\end{adjustbox}
\caption{Correspondence between the upper bounds on the propositional modal validities and control statements.}
  \label{Table.upper-bounds}
\end{table}

There is an intriguing history behind these new results. Joel David Hamkins and I had conjectured that buttons suffice to bound the propositional modal validities by \theoryf{S4.2}~\cite{Hamkins2019:X-buttons-suffice}. However, Hamkins retracted this claim shortly after at the Wijsgerig Festival DRIFT in Amsterdam. Meanwhile, a MathOverflow discussion has emerged, which eventually led to a concrete counterexample~\cite{Passmann19-are-buttons-really-enough-to-bounds-validities-by-S4.2}. And so, the question that remained was what happens in the absence of switches?

\begin{question}\label{Question.No-switches}
    Suppose $w$ is an initial world in a Kripke model $M$. What are the upper bounds on the propositional modal validities at $w$ in $M$ under the following conditions:
    \begin{enumerate}
        \item If $w$ admits ratchets of arbitrarily large finite lengths?
        \item If $w$ admits arbitrarily large finite families of independent buttons?
        \item If $w$ admits a labeling for every finite tree?
    \end{enumerate}
\end{question}

Let us now proceed to demonstrate that the answer to each case in question~\ref{Question.No-switches} is precisely as conjectured in table~\ref{Table.upper-bounds}. By doing so, we demonstrate the~\nameref{UpperBoundsTheorem} stated at the beginning of the article.

\subsection{Control statement theory for the modal logic \normalsize G\scriptsize RZ}

In \cite{Hamkins:The-modal-logic-of-arithmetic-potentialism}, Hamkins established that a labeling of every finite pre-tree yields the modal logic \theoryf{S4} as the upper bounds on propositional modal validities. I shall now demonstrate that having a labeling for every finite tree results in the upper bounds constituting the modal logic \theoryf{Grz}. And in fact, it suffices to consider only regular finite trees.

\medskip

The modal logic \theoryf{Grz} corresponds to the class of all antiwellfounded partial orders, that is reflexive and transitive frames of fatness one satisfying converse wellfoundedness. Because it has the finite frame property, it is characterized by the set of all finite partial orders. A standard application of the tree unraveling technique, as detailed in the proof of~\cite[Theorem 2.19]{ChagrovZakharyaschev1997:ModalLogic}, shows that \theoryf{Grz} is also characterized by the set of all finite trees. Furthermore, by introducing dummy copies of the worlds during the unraveling process, one can ensure that the resulting tree is regular.

The theorem below is optimal. The failure of labeling for a particular regular finite tree $T$ is expressible in the language of propositional modal logic and is not provable in the modal logic \theoryf{Grz}.

\begin{theorem}\label{Theorem.railyard-then-upper-bound-Grz}
    Suppose $u_0$ is an initial world in a Kripke model $N$ that admits a labeling for every regular finite tree $T$. Then, the propositional modal assertions valid at $u_0$ are contained in the modal logic \theoryf{Grz}.
\end{theorem}

\begin{proof}
    Suppose \theoryf{Grz} does not prove an assertion $\varphi$. The modal logic \theoryf{Grz} is characterized by the set of all finite regular trees. So, there is a Kripke model $M$ with an initial world $w_0$ based on a regular finite tree such that $w_0$ does not satisfy $\varphi$ in $M$. By the hypothesis and the~\nameref{FrameLabelingLemma}, there is a model labeling of $M$ for $u_0$ in $N$. That is, there is an assignment of propositional variables to modal assertions $p \mapsto \psi_p$ such that
    \begin{equation*}
        (M,w_0) \satisfies \varphi(p_0,\ldots,p_k) \quad \text{if and only if} \quad (N, u_0) \satisfies \varphi(\psi_{p_0}, \ldots, \psi_{p_k}).
    \end{equation*}
     We have that $w_0$ does not satisfy the assertion $\varphi(p_0,\ldots,p_k)$ in $M$, and so $u_0$ does not satisfy $\varphi(\psi_{p_0}, \ldots, \psi_{p_k})$ in $N$. That means $\varphi$ cannot be valid at $u_0$. And since $\varphi$ was chosen arbitrarily, every modal assertion not provable in $\theoryf{Grz}$ cannot be valid at $u_0$ in $N$.
\end{proof}

We have resolved case (3) of question~\ref{Question.No-switches}. Before addressing the second case, I should like to clarify that removing the requirement for switches in the control statements giving the upper bounds of \theoryf{S4}, specifically from pre-tree labelings, leads to the upper bounds of \theoryf{Grz}.

I claim that arbitrary finite pre-tree labelings are effectively arbitrary finite regular tree labelings independent of arbitrarily large finite families of independent switches. Let me make this more precise.

A \emph{pre-tree} is derived from a tree by replacing each node with a cluster of one or more equivalent nodes. A Kripke model based on a finite pre-tree is bisimilar to a finite \emph{regular} pre-tree, by which I mean that each cluster is of the same finite size and the induced finite tree is regular.

Thus, since all clusters in a regular pre-tree are of the same size, labeling the regular finite pre-tree is equivalent to labeling the induced tree and separately labeling the nodes within the cluster (all clusters are actually copies of each other).

A \emph{switch} is a statement such that both $\possible \varphi$ and $\possible \neg \varphi$ are necessarily true. A family of \emph{independent} switches means that its switches do not interfere with each another, meaning that flipping the switches in a subfamily does not flip the switches outside of it. Similarly and more generally, two families of control statements are \emph{independent} if they do not interfere with one another.

When all clusters are of size $n$, it is essentially a bunch of copies of a single cluster, and having an independent family of $n$ switches is equivalent to labeling the nodes within the cluster. Thus, providing labelings for every finite pre-tree is equivalent to providing labelings for every finite regular tree independent of the simultaneously provided arbitrarily large finite families of independent switches. We can therefore draw the following observation.

\begin{observation}
    Suppose $u_0$ is an initial world in a Kripke model $N$ that admits a labeling for every regular finite tree $T$ independent of an arbitrarily large finite family of independent switches. Then, the propositional modal assertions valid at $u_0$ are contained in the modal logic \theoryf{S4}. 
\end{observation}

As a result, removing the requirement for switches when asserting that the upper bounds on validities are \theoryf{S4} changes the upper bounds to \theoryf{Grz}.

\subsection{Control statement theory for the modal logic \normalsize G\scriptsize RZ\normalsize .2}

It has long been known that arbitrarily large finite families of independent buttons and switches give the upper bounds of \theoryf{S4.2} on the modal system~\cite{HamkinsLoewe2008:TheModalLogicOfForcing}. Let me now argue that arbitrarily large finite families of independent buttons are enough to bound the validities by the modal logic \theoryf{Grz.2}, and that is optimal.

\medskip

A \emph{button} is an assertion that is necessarily possibly necessary. Buttons \(b_i\), for \(i < n\), are independent at an initial world \(u_0\) of a Kripke model \(N\) if none of the buttons is necessary at \(u_0\) and necessarily, any of the buttons can be \emph{pushed} without affecting the other buttons. Formally, for an \(A \subseteq n\), we define \(\Theta_A = \bigwedge_{i \in A} \necessary b_i \land  \bigwedge_{i \in n \setminus A} \neg \necessary b_i\). This asserts that the button pattern is specified exactly by \(A\). The buttons are independent at \(u_0\) if
\begin{equation*}
    (N,u_0) \satisfies \bigwedge_{i < n} \neg \necessary b_i \land \bigwedge_{A \subseteq n} \necessary \big(  \Theta_A \implies \bigwedge_{A \subseteq A^{+} \subseteq n} \possible\Theta_{A^{+}} \big).
\end{equation*}
The buttons are all unpushed at first, at \(u_0\), and necessarily, any larger button pattern is realizable.

\medskip

The lemma below slightly improves upon~\cite[Lemma 7.3]{HamkinsLoewe2008:TheModalLogicOfForcing}, by requiring only as many independent buttons as there are nodes in the Kripke frame, rather than arbitrarily large finite families of independent buttons.

\begin{lemma}\label{Lemma.buttons-and-lattices}
    Suppose $F$ is a Kripke frame that is a finite lattice with an initial node $w_0$ and $N$ a Kripke model that validates the modal logic \theoryf{S4}. If $u_0$ is a world in $N$ that admits at least $(|F|-1)$-many independent buttons, then there exists a frame labeling of $F$ for $u_0$ in $N$.
\end{lemma}

\begin{proof}
    We base our argument on~\cite[Lemma 7.3]{HamkinsLoewe2008:TheModalLogicOfForcing}. Consider the correspondence $w \mapsto b_w$ between the nodes of $F \setminus \singleton{w_0}$ and the independent buttons at $u_0$ in $N$. For an $S \subseteq F\setminus\singleton{w_0}$, define
    \begin{equation*}
        b_S = \big( \bigwedge_{s \in S} \necessary b_s \big) \land \big( \bigwedge_{s \in F \setminus (\singleton{w_0}\cup S)} \neg \necessary b_s \big),
    \end{equation*}
    which expresses that only the buttons in $S$ are pushed. For any $w \in F$, let
    \begin{equation*}
        \Phi_w = \bigvee \set{b_S \mid w = \sup S}.
    \end{equation*}
    
    We now show that $w \mapsto \Phi_w$ is a frame labeling of $F$ for $u_0$. For every world $u$ in $N$ accessible from $u_0$, let
    \begin{equation*}
        S(u)=\set{s\in F\setminus\singleton{w_0}\mid (N,u)\satisfies \necessary b_s}.
    \end{equation*}
    Then $b_{S(u)}$ holds at $u$. So $(N, u) \satisfies \Phi_w$ if and only if $w = \sup S(u)$. And so, precisely one $\Phi_w$ can hold for each such $u$. In particular, $(N,u_0) \satisfies \Phi_{w_0}$, where no buttons are pushed yet. The first and third conditions of the definition of frame labeling are therefore satisfied.

    For the forward direction of the second condition, suppose $(N,u) \satisfies \possible \Phi_{w'}$. There exists a world $u'$ accessible from $u$ such that $(N,u') \satisfies \Phi_{w'}$. Let $S = S(u)$ and $S' = S(u')$, and let $w = \sup S$ and $w' = \sup S'$. The buttons pushed at $u$ remain so in all accessible worlds, hence $S \subseteq S'$. As a result, $w = \sup S \leq \sup S' = w'$.
    
    Conversely, suppose that $(N,u) \satisfies \Phi_w$, and let $S = S(u)$, so that $w = \sup S$. Consider $w' \geq_F w$. If $w'=w_0$, then $w=w_0$, so $(N,u)\satisfies \Phi_{w_0}$. By reflexivity, $(N,u) \satisfies \possible \Phi_{w_0}$. Otherwise, by independence of the buttons, pushing only $b_{w'}$ at $u$ means accessing some $u'$ where precisely the buttons indexed by the set $S \cup \singleton{w'}$ are pushed, that is, $(N,u')\satisfies b_{S\cup\singleton{w'}}$. But $w = \sup S \leq_F w'$, and so $w' = \sup (S \cup \singleton{w'})$. That means $(N,u') \satisfies \Phi_{w'}$, hence $(N,u) \satisfies \possible \Phi_{w'}$. 
\end{proof}

Finally, we are ready to show that arbitrarily large finite families of independent buttons suffice to bound the propositional modal validities by \theoryf{Grz.2}.

\begin{theorem}\label{Theorem.buttons-then-upper-bound-Grz.2}
    Suppose $u_0$ is an initial world in a Kripke model $N$ that admits arbitrarily large finite families of independent buttons. Then, the propositional modal assertions valid at $u_0$ are contained in the modal logic \theoryf{Grz.2}.
\end{theorem}

\begin{proof}
    Parts of the proof are reminiscent of~\cite[Lemma 9]{HamkinsLoewe2008:TheModalLogicOfForcing}. Suppose $\theoryf{Grz.2}$ does not prove an assertion $\varphi$. By lemma~\ref{Lemma.not-part-of-Grz.2-fails-in-baled-tree}, there is a Kripke model $M$ with an initial world $w_0$ based on a Kripke frame that is a finite lattice such that $w_0$ does not satisfy $\varphi$ in $M$. By the~\nameref{FrameLabelingLemma} and lemma~\ref{Lemma.buttons-and-lattices}, there is a model labeling $p \mapsto \psi_p$ of $M$ for $u_0$ in $N$. That is,
    \begin{equation*}
        (M,w_0) \satisfies \varphi(p_0,\ldots,p_k) \quad \text{if and only if} \quad (N,u_0) \satisfies \varphi(\psi_{p_0},\ldots,\psi_{p_k}). 
    \end{equation*}
    But $(M,w_0) \satisfies \neg \varphi$, so $(N,u_0) \satisfies \neg \varphi(\psi_{p_0},\ldots,\psi_{p_k})$. And so, $\varphi$ cannot be valid at $u_0$. As $\varphi$ was chosen arbitrarily, any modal assertion not provable in \theoryf{Grz.2} cannot be valid at $u_0$.
\end{proof}

Theorem~\ref{Theorem.buttons-then-upper-bound-Grz.2} is the best possible---one can express in the language of propositional modal logic that there are no $n$ independent buttons, and this assertion is not provable in \theoryf{Grz.2}.

\subsection{Control statement theory for the modal logic \normalsize G\scriptsize RZ\normalsize .3}

If there are arbitrarily long finite ratchets mutually independent with arbitrarily large finite families of
independent switches, then the upper bounds on the propositional modal validities are \theoryf{S4.3}. I should like to show that if we remove switches, that is if we just have arbitrarily long finite ratchets, then the upper bounds constitute the modal logic \theoryf{Grz.3}.

\medskip

The relevant control statements are built from buttons. Namely, buttons $r_i$, where $i<n$, form a \emph{ratchet} of length $n$ if:
\begin{enumerate}
    \item Initially only the first button, $r_0$, is pushed.
    \item Necessarily, pushing $r_j$ pushes all $r_i$ with $i<j$.
    \item Necessarily, if $r_i$ is unpushed, it is possible to push it without pushing any $r_j$ with $j>i$.
\end{enumerate}
For $i<n-1$, pushing the button $r_i$ (and hence all the preceding unpushed buttons) but not $r_{i+1}$ increases the ratchet \emph{volume} to $i+1$ (so in our ratchets the volume ranges over $1,\dots,n$ since $r_0$ is initially pushed). By definition, the ratchet volume can go up one-by-one until it is fully cranked and reaches volume $n$.

\begin{lemma}\label{Lemma.finite-linear-order-n-ratchet}
    Suppose $F$ is a Kripke frame that is a finite linear order of length $n$ with an initial node $w_0$ and $L$ is a modal logic containing \theoryf{S4}. If $N$ is a Kripke model and $u_0$ a world in it that satisfies $L$ and has a ratchet of length $n$, then there exists a frame labeling of $F$ for $u_0$.
\end{lemma}

\begin{proof}
    Let $w_0 <_F w_1 <_F \ldots <_F w_{n-1}$ be the nodes of $F$. Suppose $r_i$, $i<n$, is a ratchet of length $n$ at $u_0$. Define $\Phi_{w_i}$ to assert that the ratchet buttons are pushed precisely up to $r_i$, that is the ratchet volume is $i+1$. At $u_0$, only $r_0$ is pushed. For $(N,u) \satisfies \Phi_{w_i}$, it is immediate that $(N,u) \satisfies \possible \Phi_{w_j}$ if and only if $w_i \leq_F w_j$. We have thus given a labeling of $F$ for $u_0$ in $N$.
\end{proof}

\begin{theorem}\label{Theorem.ratchets-then-upper-bound-Grz.3}
    Suppose $u_0$ is an initial world in a Kripke model $N$ that admits ratchets of arbitrarily large finite lengths. Then, the propositional modal assertions valid at $u_0$ are contained in the modal logic \theoryf{Grz.3}.
\end{theorem}

\begin{proof}
    We proceed akin to~\cite[Theorem 11]{HamkinsLeibmanLoewe2015:StructuralConnectionsForcingClassAndItsModalLogic}. Suppose that $\theoryf{Grz.3} \not \proves \varphi$. It is easy to show that the modal logic $\theoryf{Grz.3}$ is characterized by finite linear orders. Therefore, there exists a Kripke model $M$ with an initial world $w_0$ based on a finite linearly ordered Kripke frame, and such that $(M,w_0) \not \satisfies \varphi$. By the~\nameref{FrameLabelingLemma} and lemma~\ref{Lemma.finite-linear-order-n-ratchet}, there is a model labeling of $M$ for $u_0$ in $N$. That is, there is an assignment of propositional variables to modal assertions $p \mapsto \psi_p$ such that
    \begin{equation*}
        (M,w_0) \satisfies \varphi(p_0,\ldots,p_k) \quad \text{if and only if} \quad (N, u_0) \satisfies \varphi(\psi_{p_0}, \ldots, \psi_{p_k}).
    \end{equation*}
    We have that $(M,w_0) \satisfies \neg \varphi$, and so $(N,u_0) \satisfies \neg \varphi(\psi_{p_0}, \ldots, \psi_{p_k})$. That means $\varphi$ cannot be valid at $u_0$. And since $\varphi$ was chosen arbitrarily, every modal assertion that is not provable in $\theoryf{Grz.3}$ is not valid at $u_0$ in $N$.
\end{proof}

Like in the previous cases, theorem~\ref{Theorem.ratchets-then-upper-bound-Grz.3} is the best possible. If $u_0$ in $N$ does not admit a labeling of a ratchet of length $n$, then this is expressible by a propositional modal assertion that is not provable in the modal logic $\theoryf{Grz.3}$.

\section{\theoryf{Grz.2} is characterized by finite Boolean algebras}\label{Section.Grz.2-complete-for-finite-Boolean-algebras}

We now have the necessary tools to improve theorem~\ref{Theorem.Grz.2-is-characterized-by-finite-lattices}. In this section, we establish the~\nameref{CharacterizationTheorem}, restated below.

\begin{CharacterizationTheorem}
    The modal logic \theoryf{Grz.2} is characterized by finite Boolean algebras. That is, a propositional modal assertion $\varphi$ is provable in $\theoryf{Grz.2}$ if and only if $\varphi$ is valid in every finite Boolean algebra. 
\end{CharacterizationTheorem}

The~\nameref{CharacterizationTheorem} is in fact the best possible characterization of the modal logic \theoryf{Grz.2}, meaning that no proper subclass of the class of all finite Boolean algebras can characterize it. This is because the Kripke model $N$ produced in the lemma below has the smallest frame
supporting an independent family of $n$ buttons (for independence one needs to realize every pattern of the buttons). And the assertion that there are no $n$ independent buttons is expressible in the language of propositional modal logic and is not provable in \theoryf{Grz.2}.

\begin{lemma}\label{Lemma.n-many-buttons}
    For any $n < \omega$, there is a Kripke model $N$ with an initial world $u_0$ whose Kripke frame is a finite Boolean algebra, such that $u_0$ admits $n$-many independent buttons.
\end{lemma}

\begin{proof}
     We proceed like in~\cite[Lemma 8]{HamkinsLoewe2008:TheModalLogicOfForcing}. Take the Kripke frame to be the powerset of $n$ ordered by inclusion. This is a finite Boolean algebra, whose nodes are subsets $B$ of $n$. Let $u_0=\varnothing$ be the initial world. Now let $N$ be a Kripke model based on that Kripke frame where $b_i$ is true just in case $i \in B$. Clearly, every $b_i$ is a button, and---since whatever the button-pattern $B$ any world has, any larger pattern $B' \fo B$ is possible---they are all independent at $u_0$.
\end{proof}

The following is an analogue of~\cite[Lemma 9]{HamkinsLoewe2008:TheModalLogicOfForcing} but with no switches.

\begin{theorem}\label{Theorem.Grz.2-characterization-buttons}
    A class of Kripke frames characterizes the modal logic $\theoryf{Grz.2}$ if and only if $\theoryf{Grz.2}$ is sound with respect to that class and there exist Kripke models based on frames from the class whose initial worlds admit arbitrarily large finite families of independent buttons.
\end{theorem}

\begin{proof}
    Suppose a class of Kripke frames characterizes \theoryf{Grz.2}. By definition, \theoryf{Grz.2} must be sound with respect to that class. It is expressible in the language of propositional modal logic that no family of $n$ independent buttons exists. However, by lemma~\ref{Lemma.n-many-buttons}, this assertion is refutable in \theoryf{Grz.2}, so it cannot be valid in any class of Kripke frames that characterizes \theoryf{Grz.2}. Consequently, for each $n < \omega$, there must be a Kripke model based on a frame from the class that admits $n$ independent buttons.

    For the converse, assume there is a class of Kripke frames with respect to which \theoryf{Grz.2} is sound and for which there exist Kripke models that admit arbitrarily large finite families of independent buttons. If $\theoryf{Grz.2} \not\proves \varphi$, then by lemma~\ref{Lemma.not-part-of-Grz.2-fails-in-baled-tree}, there is a Kripke model $M$ based on a Kripke frame $F$ that is a lattice, and for an initial world $w_0$ in $M$, we have $(M,w_0) \not\satisfies \varphi$. Let $N$ be a Kripke model based on a frame from the class such that an initial world $u_0$ in $N$ admits $|F|-1$ independent buttons. By lemma~\ref{Lemma.buttons-and-lattices}, there is a labeling of $F$ for $u_0$ in $N$. Thus, by the~\nameref{FrameLabelingLemma}, there is a model labeling of $M$ for $u_0$ in $N$, meaning there exists an assignment of propositional variables to modal assertions $p \mapsto \psi_p$ such that
    \begin{equation*}
        (M,w_0) \satisfies \varphi(p_0,\ldots,p_k) \quad \text{if and only if} \quad (N,u_0) \satisfies \varphi(\psi_{p_0},\ldots,\psi_{p_k}).
    \end{equation*}
    But since $(M,w_0) \not\satisfies \varphi$, it follows that $(N,u_0) \not\satisfies \varphi(\psi_{p_0},\ldots,\psi_{p_k})$. Therefore, we have a substitution instance of $\varphi$ that fails in $N$. Since $N$ is a Kripke model based on a frame from the class, we conclude that any assertion refutable in \theoryf{Grz.2} is not valid in the class. Moreover, by soundness, every assertion in \theoryf{Grz.2} is valid in the class. Thus, the class of Kripke frames indeed characterizes \theoryf{Grz.2}.
    \end{proof}

    Let me conclude this section with a theorem that summarizes the Kripke frame characterizations for the modal logic \theoryf{Grz.2}, and which extends theorem~\ref{Theorem.Grz-2-characterization-partial}.

\begin{theorem}
    The following sets of Kripke frames characterize the modal logic $\theoryf{Grz.2}$.
    \begin{enumerate}
        \item Finite directed partial orders.
        \item Finite lattices.
        \item Finite baled trees.
        \item Finite Boolean algebras.
    \end{enumerate}
\end{theorem}

\begin{proof}
    The modal logic \theoryf{Grz.2} is valid in every finite directed partial order and, in particular, is sound with respect to all the classes of Kripke frames under consideration. It therefore suffices to show that \theoryf{Grz.2} is complete for these classes. By lemma~\ref{Lemma.not-part-of-Grz.2-fails-in-baled-tree}, every non-theorem of \theoryf{Grz.2} fails on some finite baled tree, and hence on some finite directed partial order. Thus \theoryf{Grz.2} is complete for finite directed partial orders, establishing (1).

    Every baled tree is a lattice, so to prove completeness with respect to finite lattices, it suffices to show completeness with respect to finite baled trees. By lemma~\ref{Lemma.not-part-of-Grz.2-fails-in-baled-tree}, any assertion not provable in \theoryf{Grz.2} must fail at some world in a Kripke model based on a finite baled tree, thus establishing (2) and (3).

    Finally, by lemma~\ref{Lemma.n-many-buttons}, for any $n < \omega$, there exists a Kripke model based on a finite Boolean algebra whose initial world admits $n$ independent buttons. Therefore, by theorem~\ref{Theorem.Grz.2-characterization-buttons}, the class of all finite Boolean algebras characterizes the modal logic \theoryf{Grz.2}, establishing (4).
\end{proof}

\printbibliography

\newcommand\blfootnote[1]{%
	\begingroup
	\renewcommand\thefootnote{}\footnote{#1}%
	\addtocounter{footnote}{-1}%
	\endgroup
}

\end{document}